\documentclass{amsart}
\usepackage{amsthm}

\usepackage{verbatim}  
\usepackage{amsmath}

\usepackage[english]{babel}
\usepackage{latexsym}
\usepackage{cite}
\usepackage{amssymb,amsthm,amsmath}
\usepackage{graphicx,color}
\usepackage{verbatim}
  \def\fqs{\mathbb{F}_{q^2}}

\def\fq{\mathbb{F}_q}
\def\cC{\mathcal C}

\def\cX{\mathcal X}

\def\K{\mathbb{K}}

\def\div{{\rm div}}
\def\fqpr{\mathbb{F}_{q^{'}}}

\newtheorem{lemma}{Lemma}
\newtheorem{theorem}[lemma]{Theorem}
\newtheorem{proposition}[lemma]{Proposition}
\newtheorem{definition}[lemma]{Definition}

\newtheorem{remark}[lemma]{Remark}

\begin{document}

\title[Arcs and caps from cubic curves]{Complete arcs and complete caps from cubics with an isolated double point}
\thanks{This research was supported by the Italian Ministry MIUR, 
Geometrie di Galois e strutture di incidenza,
PRIN 2009--2010, 
by INdAM, and by Tubitak Proj. Nr. 111T234.}
\author[N.Anbar, D.Bartoli, M. Giulietti, I.Platoni]{Nurdag\"ul Anbar         \and Daniele Bartoli \and Massimo Giulietti \and
        Irene Platoni %etc.
}

%\authorrunning{N.Anbar, D.Bartoli, M.Giulietti, I.Platoni} % if too long for running head

\thanks{Nurdag\"ul Anbar -
Faculty of Engineering and Natural Sciences - Sabanci University\\ Orhanli-Tuzla - 34956 Istanbul - Turkey.  
              \email{nurdagul@su.sabanciuniv.edu}}           %  \\
           
\thanks{Daniele Bartoli -
           Dipartimento di Matematica e Informatica - University of Perugia\\ Via Vanvitelli 1 - 06123 Perugia - Italy.
              \email{bartoli@dmi.unipg.it}}

\thanks{           Massimo Giulietti -
           Dipartimento di Matematica e Informatica - University of Perugia\\ Via Vanvitelli 1 - 06123 Perugia - Italy.
              \email{giuliet@dmi.unipg.it}
}

\thanks{Irene Platoni -
Dipartimento di Matematica - University of Trento\\ Via Sommarive, 14 - 38123, Povo (TN) -  Italy. \email{irene.platoni@unitn.it}}

\begin{abstract} Small complete arcs and caps in Galois spaces over finite fields $\fq$ with characteristic greater than $3$ are constructed from cubic curves with an isolated double point. For $m$ a divisor  of $q+1$, 
complete plane arcs of size approximately $q/m$ are obtained, provided that $(m,6)=1$ and 
$m<\frac{1}{4}q^{1/4}$. If in addition $m=m_1m_2$ with $(m_1,m_2)=1$, then 
complete caps of size approximately
$\frac{m_1+m_2}{m}q^{N/2}$
in affine spaces of dimension $N\equiv 0 \pmod 4$ are constructed.
\end{abstract}

\maketitle

\section{Introduction}

In an (affine or projective) space over a 
finite field, a cap  is a set of points no
three of which are collinear. A cap is said to be complete if it is maximal with respect to set theoretical inclusion. Plane caps are usually called arcs.

Arcs and caps have
%and their generalizations, such as unitals, $(k,d)$-arcs and caps, have
played an important role in Finite Geometry since the pioneering work by B. Segre \cite{MR0238846}.
These objects are relevant also in  Coding Theory, being the geometrical counterpart of distinguished types of error-correcting and covering linear codes.  In this direction, an important issue is to ask for explicit constructions of small complete  caps in Galois spaces.
In fact, complete caps correspond to quasi-perfect linear codes with covering radius $2$, so that the smaller is the size of the cap, the better is the density of the covering code.

The trivial lower bound for the size of a complete cap in a Galois space of dimension $N$ and order $q$ is 
\begin{equation}\label{TrivialLB}
\sqrt 2  q^{(N-1)/2}.
\end{equation}
If $q$ is even and $N$ is odd, such bound is substantially sharp; see \cite{MR1395760}. Otherwise, all known infinite families of complete caps have size far from \eqref{TrivialLB}; see the survey papers \cite{HS1,HirschfeldStorme2001} and the more recent works \cite{ABGP,Anb-Giu,BarGiuPrep,DAOS,MR2656392,Gjac,Gjcd,GPIEE}.
For $q=p^s$ with $p>3$ a prime, the smallest explicitly described complete plane caps 
%and caps arise from cubic curves: 
are due to Sz\H onyi,
who constructed complete arcs with roughly $q/m$ points for any
 divisor $m$ of $q-1$ satisfying  $m<
\frac{1}{C}
 q^{1/4}$, 
with $C>1$ is a constant independent of $q$
% Sz\H onyi constructed 
%complete arcs of size of the same order of magnitude as $q/m$  are constructed in 
\cite{MR1221589,TamasRoma}\footnote{The condition of $m$ being a divisor of $q-1$ was not originally required in \cite{TamasRoma}, but is actually needed in order for the proof of a key lemma by Voloch to be correct; see Remark 4 in \cite{Anb-Giu}.}
. From these arcs, by using some lifting methods, complete caps of size roughly $q^{N/2}/\sqrt{m}$ in $AG(N,q)$, $N\equiv 0 \pmod 4$, are obtained in \cite{ABGPnodal,Anb-Giu}. 
The aim of this paper is to obtain similar results for the case where  $m$ is a divisor of $q+1$, in order to significantly widen the range of $q$'s for which
complete arcs in $AG(2,q)$ of size about $q^{3/4}$, as well as complete caps in $AG(N,q)$ with roughly $q^{(4N-1)/8}$ points, can actually be constructed.
To this end, plane cubics with an isolated double points are considered.

%For a given $q$, such results provide  complete arcs in $AG(2,q)$ of size about %$q^{3/4}$, as well as complete caps in $AG(N,q)$ with roughly $q^{(4N-1)/8}$ points, %only when $q-1$ admits a divisor close to $q^{1/4}$.
%The aim of this paper is to obtain results similar to those in %\cite{ABGPnodal,TamasRoma,MR1221589} for the case where  $m$ is a divisor of $q+1$, %in order to significantly widen the range of $q$'s for which complete arcs and caps %with size of the above orders of magnitude actually exist. To this end, plane cubics %with an isolated double points are considered. 

Let $G$ denote 
the abelian group  of the non-singular $\fq$-rational points of an irreducible plane cubic $\cX$ defined over $\fq$. It was already noted by Zirilli \cite{ZIR} that no three points in a coset $A$ of a subgroup $K$ of $G$  can be collinear, provided that the index $m$ of $K$ in $G$ is not divisible by $3$. Since then, arcs in cubics have been thoroughly investigated, as well as caps arising from these arcs by recursive constructions; see \cite{ABGP,ABGPnodal,Anb-Giu,FaPaSc,GFFA,HV,MR792577,MR1221589,TamasRoma,VOL,MR1075538}. 
However, no results about arcs and caps from cubics with an isolated double point have appeared so far. One of the problems that come up when dealing with these cubics  is that the natural parametrization of the points of $A$, arising from the natural isomorphism between $G$ and the subgroup of order $q+1$ of the multiplicative group of $\mathbb F_{q^2}$, involves  polinomial functions defined over $\mathbb F_{q^2}$ but not over $\fq$. This makes it impossible a straightforward application of the classical method by 
Segre \cite{MR0149361} and Lombardo Radice \cite{LombRad}
for proving that a point $P$ off $\cX$ is collinear with two points in $A$; in fact, such method needs that the algebraic curve $\cC$ describing the collinearity with $P$ and two generic points in $A$ is defined over $\fq$.  A key point of the paper is to  overcome  such a difficulty by finding a curve which is birationally equivalent to $\cC$, but is defined over $\fq$; see Lemmas \ref{irrmodificato} and \ref{troppe}.

The main achievements here are Theorems \ref{unouno} and \ref{duedue}.
For a divisor $m$ of $q+1$ such that $(m,6)=1$ and $m\le \sqrt[4]{q}/4$,
we explicitly describe a complete arc of size approximately $m+\frac{q+1}{m}$; 
if in addition $m$ admits a non-trivial factorization $m=m_1m_2$ with $(m_1,m_2)=1$, we also provide 
complete caps  of size approximately $\frac{m_1+m_2}{m}q^{N/2}$ in affine spaces $AG(N,q)$ with dimension $N\equiv 0 \pmod 4$.

The paper is organized as follows.
 In Section 2 we review some of the standard facts on curves and algebraic  function fields. We also briefly sketch a recursive construction from \cite{Gjac} of complete caps from bicovering arcs, that is arcs for which completeness holds in a stronger sense; see Definition \ref{bico}.  Section 3 presents some preliminary results on the algebraic curve describing the collinearity with $P$ and two generic points in $A$. The proof that under our assumptions on $m$ almost each point $P$ not on $\cX$ is bicovered by the secants of $A$ is  the main object of Section 4; see Propositions \ref{archinuovi}, \ref{archinuoviVAR}, and \ref{archibishop}. The case where $P$ lies in $\cX$ is dealt with in Proposition \ref{bicointerni}. Finally, the proof of our main results is completed in Section 5.

\section{Preliminaries}
Let $q$ be an odd prime power, and let $\fq$ denote the finite field with $q$ elements. Throughout the paper,  $\K$ will denote the algebraic closure of $\fq$.

\subsection{Curves and function fields}

%Let $\K$ denote an algebraically closed field and 

Let $\cC$ be a projective absolutely irreducible algebraic curve, defined over the algebraic closure $\K$ of $\fq$. 
%We recall the definition of a function field over $K$ and of its full constant field.
An \textit{algebraic function field $F$ over $\K$} is an extension $F$ of $\K$ such that $F$ is a finite algebraic extension of $\K(x)$, for some element $x\in F$ transcendental over $\K$. If $F=\K(x)$, then $F$ is called the \textit{rational function field over $\K$}.
%The \textit{full constant field of $F$} (also called the \textit{field of constants of $F$}) is the finite %extension $\tilde{K}$ of $K$, consisting of the elements in $F$ that are algebraic over $K$.
%
%We can equivalently say that $K$ is algebraically closed in $F$ or that $K$ is the full constant field of $F$ %when $K=\tilde{K}$.
For basic definitions on function fields we refer to \cite{STI}.

It is well known that to any  curve $\cC$ defined over $\K$ one can associate a function field $\K(\cC)$ over $\K$, namely the field of the rational functions of $\cC$.
Conversely, to a  function field $F$ over $\K$ one can associate a curve $\cC$, defined over $\K$, such that $\K(\cC)$ is $\K$-isomorphic to $F$. The genus of $F$ as a function field coincides with the genus of $\cC$.

A place $\gamma$ of $\K(\cC)$ can be associated to a single point of $\cC$ called the \textit{center} of $\gamma$, but not vice versa. A  bijection between places of $\K(\cC)$ and points of $\cC$ holds provided that the curve $\cC$ is non-singular. 
%We will identify points of $\cC$ with places of $\K(\cC)$, when it is allowed.

Let $F$ be a function field over $\K$. If $F'$ is a finite extension of  $F$, then a place $\gamma '$ of $F'$ is said to be \textit{lying over} a place $\gamma$ of $F$
, %and we write $\gamma'|\gamma$, 
if $\gamma\subset \gamma '$. This holds precisely when $\gamma = \gamma ' \cap F$. In this paper $e\left(\gamma ' | \gamma \right)$ will denote the \textit{ramification index} of $\gamma '$ over $\gamma$.
%, and $f\left(\gamma'| \gamma \right)$ the \textit{relative degree} of $\gamma'$ over $\gamma$, that is the %degree of the extension of the residue class field of $\gamma'$ over the residue class field of $\gamma$.
%If $F$ is a function field over $\fq$, then a rational place $\gamma$ of $F$ is said to {\em split completely} over $F'$ if $e(\gamma'|\gamma)=f(\gamma'|\gamma)=1$ for each $\gamma'$ lying over $\gamma$.

A finite extension $F'$ of a function field $F$ is said to be {\em unramified} if $e(\gamma'|\gamma)=1$ for every  $\gamma'$
place of $F'$ and every $\gamma$ place of $F$ with  $\gamma'$ lying over $\gamma$.

%If $F_1$ and $F_2$ are finite extensions of a function field $F$ over $\fq$, then the \textit{compositum} $F_1 F_2$ is the subfield of the algebraic closure of $F$ generated by $F_1$ and $F_2$. It is possible that the full constant field of $F_1 F_2$ is a proper extension of $\fq$, even when $\fq$ is the full constant field of both $F_1$ and $F_2$. In order to investigate the full constant field of the compositum of two function fields, the following results can be useful.

\begin{proposition}[Proposition 3.7.3 in \cite{STI}]\label{teo1}
%Let $\K$ be an algebraically closed field.
Let $F$ be an algebraic function field over $\K$, and let $m>1$ be an integer relatively prime to the characteristic of $\K$. Suppose that $u\in F$ is an element satisfying
$
u \neq \omega^e\mbox{ for all }\omega \in F \mbox{ and } e|m\mbox{, }e>1.
$
Let
\begin{equation}
 F'=F(y)\mbox{ with }y^m=u.
\end{equation}
Then
\begin{itemize}
 % \item [i)] The polynomial $\phi(T)=T^n-u$ is the minimal polymonial of $y$ over $F$ (in particular, it is irreducible over $F$). The extension $F'/F$ is Galois of degree $n$; its Galois group is cyclic, and all automorphisms of $F'/F$ are given by $\sigma(y)=\xi y$, where $\xi\in K$ is an $n$-th root of unity.
\item[(i)] for $\gamma'$ a place of $F'$ lying over a place $\gamma$ of $F$, we have
$
e(\gamma'| \gamma)=\frac{m}{r_\gamma}
%\hspace{1cm}\mbox{and}\hspace{1cm}d(\gamma'| \gamma)=\frac{m}{r_P}-1,
$
where
\begin{equation}\label{eq55}
 r_\gamma:=(m,v_{\gamma}(u))>0
\end{equation}
is the greatest common divisor of $m$ and $v_\gamma(u)$;
\item[(ii)] if  $g$ (resp. $g'$) denotes the genus of $F$ (resp. $F'$) as a function field over $\K$, then
$$
g'=1+m\left( g-1+\frac{1}{2}\displaystyle\sum_{\gamma} \left(1-\frac{r_\gamma}{m}\right) \right),
$$
where $\gamma$ ranges over the places of $F$ and $r_\gamma$ is defined by \eqref{eq55}. 
\end{itemize}
\end{proposition}
An extension such as $F'$ in Proposition \ref{teo1} is said to be a {\em Kummer extension} of $F$.

%Now let $\K$ be the algebraic closure of $\fq$.

 A curve $\cC$ is said to be defined over $\fq$ if the ideal of $\cC$ is generated by polynomials with coefficients in $\fq$. 
In this case, $\fq(\cC)$ denotes the subfield of $\K(\cC)$ consisting of the rational functions defined over $\fq$.
% and it is a rational function field over $\fq$. 
%Also, the places of $\fq(\cC)$ correspond to the orbits of points of $\cC$ under the Frobenius map on $\fq$. 
A place of $\K(\cC)$ is said to be $\fq$-rational if it is fixed by the Frobenius map on $\K(\cC)$. The center of an $\fq$-rational place is an  $\fq$-rational point of $\cC$; conversely, if $P$ is a simple $\fq$-rational point of $\cC$, then the only place centered at $P$ is $\fq$-rational. 
The following result is a corollary to Proposition \ref{teo1}.
\begin{proposition}\label{teo1cor}
Let $\cC$ be an irreducible plane curve of genus $g$ defined over $\fq$. Let $u\in \fq(\cC)$ be a non-square in $\K(\cC)$. 
Then the Kummer extension $\K(\cC)(w)$, with $w^2=u$, is the function field of some irreducible curve defined over $\fq$ of genus
$$
g'=2g-1+\frac{M}{2},
%\displaystyle\sum_{\gamma\in \mathbb{P}_F}\left(1-\frac{r_\gamma}{m}\right) \right),
$$
where $M$ is the number of places of $\K(\cC)$ with odd valuation of $u$.
\end{proposition}
The function field $\K(\cC)(w)$ as in Proposition \ref{teo1cor} is said to be a {\em double cover} of $\K(\cC)$ (and similarly the corresponding irreducible curve defined over $\fq$ is called a double cover of $\cC$).

Finally, we recall the Hasse-Weil bound, which will play a crucial role in our proofs.
\begin{proposition}[Hasse-Weil Bound - Theorem 5.2.3 in \cite{STI}]\label{HaWe}
The number $N_q$ of $\fq$-rational places of the function field $\K(\cC)$ of a curve $\cC$ defined over $\fq$ with genus $g$ satisfies
$$
|N_q-(q+1)|\le 2g \sqrt q.
$$
\end{proposition}

% Let $\mathbb P_{\mathcal C}$ denote the set of all places of $\K(\mathcal C)$, and let
%$ \mathrm{Div}(\K(\mathcal{C}))$ be the group of divisors of $\K(\mathcal C)$, that is the free abelian group generated by $\mathbb P_{\mathcal C}$.

\subsection{Complete caps from bicovering arcs}

Throughout this section, $N$ is assumed to be a positive integer
divisible by $4$. Let $q'=q^\frac{N-2}{2}$. Fix a basis of $\fqpr$
as a linear space over $\fq$, and identify points in $AG(N,q)$
with vectors of $\fqpr\times \fqpr \times \fq \times \fq$. 
%Note that as $\frac{N-2}{2}$ is
%odd, $\tau$ is a non-square in $\fqpr$ as well.

For an arc ${\mathcal A}$ in $AG(2,q)$, let
% Let $K_2\times A$ denote the
%product of the cap
%$$K_2=\{(\alpha,\alpha^2)\mid \alpha \in \fqpr\}$$ in
%$AG(N-2,q)$ by an arc $A$ in $AG(2,q)$. Explicitly,
$$C_{\mathcal A}=\{(\alpha,\alpha^2,u,v)\in AG(N,q)\mid \alpha \in
\fqpr\,,\,\,\,(u,v)\in {\mathcal A}\}\,.$$
As noticed in \cite{Gjac}, the set $C_{\mathcal A}$ is a cap whose completeness in $AG(N,q)$ depends on the bicovering properties of ${\mathcal A}$ in $AG(2,q)$, defined as follows.
According to Segre \cite{MR0362023}, given three pairwise
distinct  points $P,P_1,P_2$ on a line $\ell$ in $AG(2,q)$, $P$ is
external or internal to the segment $P_1P_2$ depending on whether
\begin{equation}\label{exto}
(x-x_1)(x-x_2)\quad \text{is a non-zero square  or
a non-square in }\fq,
\end{equation}
 where $x$,
$x_1$ and $x_2$ are the coordinates of $P$, $P_1$ and $P_2$ with
respect to any affine frame of $\ell$. 
\begin{definition}\label{bico}
Let ${\mathcal A}$ be a complete arc in $AG(2,q)$. A point $P\in AG(2,q)\setminus {\mathcal A}$ is said to be bicovered by ${\mathcal A}$ if there exist $P_1,P_2,P_3,P_4\in {\mathcal A}$ such that
$P$ is both external to the segment $P_1P_2$ and internal to the segment $P_3P_4$. If every $P\in AG(2,q)\setminus {\mathcal A}$ is bicovered by ${\mathcal A}$, then ${\mathcal A}$ is said to be a bicovering arc. If there exists precisely one point $Q\in AG(2,q)\setminus {\mathcal A}$ which is not bicovered by ${\mathcal A}$, then ${\mathcal A}$ is said to be almost bicovering, and $Q$ is called the center of ${\mathcal A}$.
\end{definition}

A key tool in this paper is the following result from \cite{Gjac}.
\begin{proposition}\label{mainP}
Let $\tau$ be a non-square in $\fq$. 
If ${\mathcal A}$ is a bicovering $k$-arc, then $C_{\mathcal A}$ is a complete cap in $AG(N,q)$ of size $kq^{(N-2)/2}$. If ${\mathcal A}$ is almost bicovering with center $Q=(x_0,y_0)$, then either
$$C=C_{\mathcal A} \cup \{(\alpha,\alpha^2-\tau,x_0,y_0)\mid \alpha \in
\fqpr\}$$ 
or
$$C=C_{\mathcal A} \cup \{(\alpha,\alpha^2-\tau^2,x_0,y_0)\mid \alpha \in
\fqpr\}$$ 
is a complete cap in $AG(N,q)$ of  size $(k+1)q^{(N-2)/2}$. The former case occurs precisely when $Q$ is external to every secant of ${\mathcal A}$ through $Q$.
\end{proposition}

\section{A family of curves defined over $\fq$}\label{secfam}

Throughout this section $q=p^h$ for some prime $p>3$, and $m$ is a proper divisor of $q+1$ with $(m,6)=1$. Also, $\bar t$ is a non-zero element in $\fqs$ which is not an $m$-th power in $\fqs$. Let $A,B\in \fqs$ with $AB\neq (A-1)^3$. An important role for the present investigation is played by the curve
\begin{equation}\label{curva2bis}
\cC_{A,B,\bar t,m}: f_{A,B,\bar t,m}(X,Y)=0, 
\end{equation}
where
\begin{equation}\label{curva2}
\begin{array}{rcl}
f_{A,B,\bar t,m}(X,Y)& = &A(\bar t^3X^{2m}Y^m+\bar t^3X^mY^{2m}-3\bar t^2X^{m}Y^m+1) 
-B\bar t^2X^mY^m\\ & & -\bar t^4X^{2m}Y^{2m}+3\bar t^2X^mY^m-\bar tX^m-\bar tY^m.
\end{array}
\end{equation}

The curve $\cC_{A,B,\bar t,m}$ was thoroughly investigated in \cite{ABGPnodal}.

\begin{proposition}[Case 2 of Proposition 9 in \cite{ABGPnodal}]\label{P26apr13}
Let $A,B$ be such that
\begin{itemize}
\item $AB\neq (A-1)^3$;
\item $A\neq 0$;
\item  either $A^3\neq -1$ or $B\neq 1-(A-1)^3$.
\end{itemize}
Then the curve $\cC_{A,B,\bar t,m}$ is absolutely irreducible of genus $g\le 3m^2-3m+1$.
\end{proposition}

Under the assumptions of Proposition \ref{P26apr13},
let $\bar x$ and $\bar y$ denote the rational functions associated to the affine coordinates $X$ and $Y$, respectively. Then $\K(\cC_{A,B,\bar t,m})=\K(\bar{x},\bar{y})$ with $f_{A,B,\bar t,m}(\bar x,\bar y)=0$. Let $\bar u=\bar x^m$ and $\bar z=\bar y^m$.
The following results from \cite{ABGPnodal} about the function field extension $\K(\bar x,\bar y):\K(\bar u,\bar z)$ will be needed.

\begin{proposition}[Lemma 4 in \cite{ABGPnodal}]\label{clear} In the function field $\K(\bar u,\bar z)$, there exist six places $\gamma_j$, $j=1,\ldots,6$, such that
$$
\div(\bar u)=\gamma_4+\gamma_5-\gamma_1-\gamma_2, \qquad
\div(\bar z)=\gamma_2+\gamma_6-\gamma_3-\gamma_4.
$$
\end{proposition}

\begin{proposition}[Case 2 of Proposition 9 in \cite{ABGPnodal}]\label{clear2}
For each $j=1,\ldots, 6$, the ramification index of $\gamma_j$ in the extension $\K(\bar x,\bar y)$ over $\K(\bar u,\bar z)$ is equal to $m$, and no other place of $\K(\bar u,\bar z)$ is ramified.
\end{proposition}

According to  \cite{ABGPnodal}, 
for $j=1,\ldots,6$, let  ${\bar {\bar{\gamma}}}_j^1, \ldots,{\bar {\bar{\gamma}}}_j^m$ denote the  places of $\K(\bar x, \bar y)$ lying over the place $\gamma_j$ of $\K(\bar u, \bar z)$.

\begin{proposition}[Case 2 of Proposition 9 in \cite{ABGPnodal}]\label{clear3}
In $\K(\bar x,\bar y)$,
\begin{equation}\label{26apr}
\div\big((A-\bar t\bar x^m)(A-\bar t\bar y^m)\big)=m\Big(\sum_{i=1}^m({\bar{\bar {\gamma}}}_5^i+{\bar{\bar {\gamma}}}_6^i-{\bar{\bar {\gamma}}}_4^i-{\bar{\bar {\gamma}}}_2^i)\Big).
\end{equation}
\end{proposition}

In order to investigate the bicovering properties of a coset of index $m$ in the abelian group of the non-singular $\fq$-rational points of a cubic with an isolated double point
we need to establish whether 
$$\frac{(A-\bar t\bar x^m)(A-\bar t\bar y^m)}{(1-\bar t\bar x^m)(1-\bar t\bar y^m)}$$
is a square in $\K(\bar x, \bar y)$.

\begin{proposition}\label{exiisolated} Assume that $A$ and $B$ satisfy the conditions of Proposition {\rm{\ref{P26apr13}}}. 
For $d\in \K$, $d\neq 0$, let
$$
\eta=d\frac{(A-\bar t\bar x^m)(A-\bar t\bar y^m)}{(1-\bar t\bar x^m)(1-\bar t\bar y^m)}.
$$

If $A\neq 1$, then
\begin{itemize}

\item [{\rm{(i)}}] the divisor of $\eta$ is
$$
m\sum_{i=1}^m({\bar{\bar \gamma}}_5^i+{\bar{\bar \gamma}}_6^i+{\bar{\bar \gamma}}_1^i+{\bar{\bar \gamma}}_3^i)-{\bar{\bar{D}}},
$$
where ${\bar{\bar{D}}}$ is a divisor of degree $4m^2$ whose support consists of places not lying over  %places of il cui supporto \`e costituito da posti la cui restrizione a $\K(\bar u, \bar z)$ non appartiene a  
any place in $\{\gamma_j\mid j=1,\ldots,6\}$;

\item [{\rm{(ii)}}] the function field $\K(\bar x,\bar y,\bar w)$ with $\bar w^2=\eta$ is a Kummer extension of $\K(\bar x,\bar y)$;

\item [{\rm{(iii)}}]
the genus of the function field $\K(\bar x,\bar y,\bar w)$  is less than or equal to $8m^2-4m+1$.
\end{itemize}
\end{proposition}
\begin{proof}
By Propositions \ref{clear}, from $A\neq 1$ it is easy to deduce that
%
%
%From \eqref{gen21}
%$$
%\div(( t \bar x^m-a)( t \bar y^m-a))=m\sum_{i=1}^m({\bar{\bar \gamma}}_5^i+{\bar{\bar \gamma}}_6^i-{\bar{\bar \gamma}}_4^i-{\bar{\bar \gamma}}_2^i).
%$$ 
%As $a\neq 1$, it is easy to deduce that 
  the divisor of $1- \bar t \bar u$ in $\K(\bar u,\bar z)$ is
$$
-\gamma_1-\gamma_2+ D_1,
$$
where $D_1$ is the degree-$2$ divisor of the zeros of  $1- \bar t \bar u$.
%, il cui supporto sono i posti centrati nei punti affini comuni alla quartica $\mathcal Q_P$ e alla retta %$U\bar t=1$. 
Similarly,
$$
\div(1- \bar t \bar z)=-\gamma_3-\gamma_4+ D_2,
$$
and hence
in $\K(\bar u,\bar z)$ we have
$$
\div((1- \bar t\bar u)(1- \bar t\bar z))=-\gamma_1-\gamma_2-\gamma_3-\gamma_4+D,
$$
where $D$ is a divisor of degree $4$ whose support is disjoint from 
$\{\gamma_i\mid i=1,\ldots,6\}$.
Therefore, by Proposition \ref{clear2},
$$
\div((1- \bar t \bar x^m)( 1- \bar t \bar y^m))=m\sum_{i=1}^m(-{\bar{\bar \gamma}}_1^i-{\bar{\bar \gamma}}_2^i-{\bar{\bar \gamma}}_3^i-{\bar{\bar \gamma}}_4^i)+{\bar{\bar{D}}},
$$ 
where ${\bar{\bar{D}}}$ is a divisor of degree $4m^2$ whose support is disjoint from the set of places %places of il cui supporto \`e costituito da posti la cui restrizione a $\K(\bar u, \bar z)$ non appartiene a  
lying over $\{\gamma_i\mid i=1,\ldots,6\}$.
Then by Proposition \ref{clear3} the divisor of $\eta$ is
$$
m\sum_{i=1}^m({\bar{\bar \gamma}}_5^i+{\bar{\bar \gamma}}_6^i+{\bar{\bar \gamma}}_1^i+{\bar{\bar \gamma}}_3^i)-{\bar{\bar{D}}}.
$$
This proves (i). As $\eta$ is not a square in $\K(\bar x, \bar y)$, assertion (ii) holds as well. Finally, Proposition \ref{teo1} yields (iii).
\end{proof}

\section{Covering properties of certain subsets of $\cX$}\label{covvi} 

Throughout this section we fix an element $\beta$ in $\fqs\setminus \fq$ such that $\beta^2\in \fq$.
%
%
%
%Let $\theta$ be a primitive element of $\fqs$, and let $\beta=\theta^{\frac{q+1}{2}}$. Then 
%$$
%\beta \in \fqs\setminus \fq,\qquad  \beta^2\in \fq.
%$$
%Moreover, $\beta^{q}=\theta^{\frac{q^2+q}{2}}=\theta^{\frac{q^2-1}{2}}\theta^{\frac{q+1}{2}}=-\beta$.
Let $\cX$ be the plane cubic with equation
$$
Y(X^2-\beta^2)=1.
$$
The point $Y_\infty$ is an isolated double point with tangents $X=\pm \beta$, and $X_\infty$ is an inflection point with tangent $Y=0$.
We choose $X_\infty$ as the neutral element of the abelian group
$(\cX\setminus \{Y_\infty\},\oplus)$ of the non-singular points of $\cX$.

%We construct two different parametrization of the points of $\cX$.

%For an element $u \in \K \cup \{\infty\}$ we define a point $P_u\in \cX$ as follows: 
% $P_u=(u,1/(u^2-\beta^2))$ when $u\in \K\setminus \{\pm \beta\}$, $P_u=Y_\infty$ if $u=\pm \beta$, and $P_\infty=X_\infty$.

For $v \in \K \setminus \{0,1\}$, let
$Q_v$ be the point on $\cX$ with affine coordinates $\big(\frac{v+1}{v-1}\beta, \frac{(v-1)^2}{4v\beta^2}\big)$. Also, let $Q_0=Y_\infty$ and $Q_1=X_\infty$.
% \{\infty\}$ we let
%$Q_v=\big(\frac{v+1}{v-1}\beta, \frac{(v-1)^2}{4v\beta^2}\big)$ for $v\in %\K\setminus \{0,1\}$,
%$Q_v=Y_\infty$ if $v\in \{0,\infty\}$, $Q_1=X_\infty$.
Such a parametrization actually defines an isomorphism between  
$(\cX\setminus \{Y_\infty\},\oplus)$ and the multiplicative group of $\K$. In fact, it is straightforward to check that for $v,w \in \K^*$, 
\begin{equation}\label{8feb}
Q_v\oplus Q_w=Q_{vw}.
\end{equation}

The $(q+1)$ non-singular $\fq$-rational points of $\cX$ form a cyclic subgroup $G$
 of $(\cX\setminus \{Y_\infty\},\oplus)$. It is easily seen that 
$$
G=\{Q_{\frac{u+\beta}{u-\beta}}\mid u \in \fq \} \cup \{X_\infty\}.
$$
For a divisor $m$ of $q+1$, the group $G$ has precisely one subgroup $K$ of index $m$, consisting of the $m$-th powers in $G$. By \eqref{8feb},
$$
K=\big\{Q_{(\frac{u+\beta}{u-\beta})^m}\mid u \in \fq \big\} \cup \{X_\infty\}.
$$
Let $T=Q_{\bar t}$ be a point in $G\setminus K$ and let $K_T$ be the coset $K\oplus T$. 
%Let
%$$
%\bar t=\frac{ t+\beta}{ t-\beta}.
%$$
%so that $T=Q_{\bar t}$ holds.
Then
\begin{equation}\label{descri}
K_T=\big\{Q_{\bar t (\frac{u+\beta}{u-\beta})^m}\mid u \in \fq\big\} \cup \{Q_{\bar t}\}.
\end{equation}

Throughout this section $a,b$ are elements in $\fq$ with $b(a^2-\beta^2)\neq 1$, and $P$ is the point in $AG(2,q)\setminus \cX$ with affine coordinates $(a,b)$. We also assume that $(m,6)=1$.
Let
$$ 
  g_{a,b}(X,Y):=bX^2Y^2 - (b\beta^2+1) (X^2+Y^2) 
        - XY +a(X + Y) + \beta^2(b\beta^2 +1),
$$
and
$$
L_{a,b,\bar t, m}(X,Y)=(\bar tX^m-1)^2(\bar t Y^m-1)^2g_{a,b}(\beta\frac{\bar tX^m+1}{\bar tX^m-1},\beta\frac{\bar tY^m+1}{\bar tY^m-1}).
$$
\begin{lemma}\label{alldebolenuova}
Let $(x,y)$ be an affine point of the curve $L_{a,b,\bar t, m}(X,Y)=0$. If 
$$(\bar tx^m-1)(\bar t y^m-1)(x^m-y^m)\neq 0,$$ then $P$ is collinear with 
 $Q_{\bar t x^m}$ and $Q_{\bar t y^m}$.
\end{lemma}
\begin{proof}
We first note that for
 $u,v$ distinct elements in $\K\setminus \{\pm \beta\}$, the point $P$ is collinear with $(u,\frac{1}{u^2-\beta^2})$ and $(v,\frac{1}{v^2-\beta^2})$ if and only if 
$g_{a,b}(u,v)=0$. In fact,
$$\det\left( \begin{array}{ccc}
u&\frac{1}{u^2-\beta^2}&1\\
v&\frac{1}{v^2-\beta^2}&1\\
a&b&1\\
\end{array}\right)
$$
is equal to
$$\frac{1}{(u^{2}-\beta^2)(v^{2}-\beta^2)}(v-u)[bu^{2}v^{2}-(b\beta^{2}+1)(u^{2}+v^{2})-uv+a(u+v)+b\beta^4+\beta^{2}].$$
It is straightforward to check that $Q_{\bar tx^m}$ coincides with $(u,\frac{1}{u^2-\beta^2})$  precisely when $u=\beta\frac{\bar t x^m+1}{\bar t x^m-1}$.
Then the claim follows by the definition of $L_{a,b,\bar t, m}$.
\end{proof}

The curve with equation $L_{a,b,\bar t,m}(X,Y)=0$ actually belongs to the family described in Section \ref{secfam}.
\begin{lemma}\label{l55} Let 
$$
A=\frac{a+\beta}{a-\beta}, \qquad B=\frac{8b\beta^3}{a-\beta}.
$$
Then
$$
L_{a,b,\bar t, m}(X,Y)=-2\beta(a-\beta)f_{A,B,\bar t,m}(X,Y)
$$
where $f_{A,B,\bar t,m}$ is defined as in \eqref{curva2}.
\end{lemma}
\begin{proof}
The proof is a straightforward computation.
\end{proof}

Henceforth,  $\sqrt{-3}$ will denote a fixed square root of $-3$ in $\fqs$.

\begin{lemma} If
\begin{equation}\label{condizione}
(a,b)\notin \Big\{(0,-\frac{9}{8\beta^2}), (\beta \sqrt{-3},0), (-\beta \sqrt{-3},0) \Big\}
\end{equation}
then
 $L_{a,b,\bar t, m}(X,Y)=0$ is an absolutely irreducible curve with genus less than or equal to $3m^2-3m+1$. 
\end{lemma}
\begin{proof}
For $A,B$ as in Lemma \ref{l55},
let $\cC_{A,B,\bar t,m}$ be as in \eqref{curva2bis}.
By Lemma \ref{l55}, the curve $L_{a,b,\bar t, m}(X,Y)=0$ is 
actually $\cC_{A,B,\bar t,m}$.
Note that $m$ divides $q^2-1$ and that each coefficient of 
$f_{A,B,\bar t, m}(X,Y)$ lies in $\fqs$. Then by Proposition \ref{P26apr13} the curve
$\cC_{A,B,\bar t,m}$ is absolutely irreducible of genus $g\le 3m^2-3m+1$, provided that none of the following holds:

\begin{itemize}

\item[(1)] $AB=(A-1)^3$;

\item[(2)] $A=0$;

\item[(3)] $A^3=-1$ and $B=1-(A-1)^3$.
\end{itemize}
Case (1) cannot occur as $b(a^2-\beta^2)\neq 1$. Also, $a \in \fq$ implies $a+\beta \neq 0$, which rules out (2).
%
%Vediamo la prima possibilit\`a:
%$$
%\frac{a+\beta}{a-\beta}\cdot\frac{8b\beta^3}{a-\beta}=\frac{8\beta^3}{(a-\beta)^3}
%$$
%e quindi
%$$
%(a^2-\beta^2)b=1
%$$
%impossibile perch\'e $(a,b)$ non appartiene alla cubica
Assume then that (3) holds. Then $A^3=-1$ implies $a(a^2+3\beta^2)=0$.
From  $B=1-(A-1)^3$ we deduce
$$b=3\frac{a^2+3\beta^2}{8\beta^{2}(\beta a-\beta^2)}.$$
Then either $(a,b)=(0,-\frac{9}{8\beta^2})$ or $(a,b)=(\pm\beta \sqrt{-3},0)$, a contradiction.
\end{proof}
\begin{remark}\label{quattordici} Let $q=p^s$ with $p>3$ a prime. Then $-3$ is a non-square in $\fq$ if and only if $s$ is odd and $p\equiv 2 \pmod 3$;
see e.g.  {\rm \cite[Lemma 4.5]{GFFA}}.
\end{remark}

In order to show that if \eqref{condizione} holds then $P$ is collinear with two points in $K_T$, we need to ensure the existence of a point $(x,y)$ of the curve $L_{a,b,\bar t, m}(X,Y)=0$ such that $Q_{\bar tx^m}$ and $Q_{\bar t y^m}$ are distinct points in $K_T$. To this end, it is useful to consider a curve which is birationally equivalent to $L_{a,b,\bar t, m}(X,Y)=0$, but, unlike $L_{a,b,\bar t, m}(X,Y)=0$, is defined over $\fq$. 

Let
$$
M_{a,b,\bar t, m}(R,V):=(R-\beta)^{2m}(V-\beta)^{2m}L_{a,b,\bar t,m}\Big(\frac{R+\beta}{R-\beta},\frac{V+\beta}{V-\beta}\Big)=0.
$$

% 
%NON SOLO: SE IL PUNTO $(R,V)$ \`E $\fq$-RAZIONALE ALLORA I PUNTI $Q_{\bar t (\frac{R+\beta}{R-\beta})^m}$ E %$Q_{\bar t (\frac{V+\beta}{V-\beta})^m}$ STANNO AUTOMATICAMENTE IN $K_T$.

\begin{lemma}\label{irrmodificato}  If \eqref{condizione} holds, then
$M_{a,b,\bar t,m}(R,V)=0$ is an absolutely irreducible curve birationally equivalent to $L_{a,b,\bar t,m}(X,Y)=0$.
% with genus less than or equal to $3m^2-3m+1$.
\end{lemma}
\begin{proof}
Let $\K(\bar x,\bar y)$ be the function field of $L_{a,b,\bar t,m}(X,Y)=0$, so that $L_{a,b,\bar t,m}(\bar x,\bar y)=0$. Both the degrees of the extensions $\K(\bar x,\bar y):\K(\bar x)$ and $\K(\bar x,\bar y):\K(\bar y)$ are equal to $2m$. Let
$$
\bar r:=\beta \frac{\bar x+1}{\bar x-1},\qquad \bar v:=\beta\frac{\bar y+1}{\bar y-1}.
$$
Then $M_{a,b,\bar t,m}(\bar r,\bar v)=0$. As
$$
\bar x= \frac{\bar r+\beta}{\bar r-\beta},\qquad \bar y=\frac{\bar v+\beta}{\bar v-\beta}
$$
we have
$$
\K(\bar x,\bar y)=\K(\bar r,\bar v),\qquad \K(\bar x)=\K(\bar r),\qquad \K(\bar y)=\K(\bar v).
$$
Therefore, both the degrees of the extensions $\K(\bar r,\bar v):\K(\bar r)$ and $\K(\bar r,\bar v):\K(\bar v)$ are equal to $2m$. 
As the degrees of  $M_{a,b,\bar t,m}(R,V)$ in both $R$ and  $V$ are also equal to $2m$, the polynomial $M_{a,b,\bar t, m}(R,V)$ cannot be reducible.
\end{proof}

\begin{lemma}\label{troppe} The curve with equation $M_{a,b,\bar t,m}(R,V)=0$ is defined over $\fq$.
\end{lemma}
\begin{proof}
We are going to show that up to a scalar factor in $\K^*$ the coefficients of 
$M_{a,b,\bar t,m}(R,V)$ lie in $\fq$.
Consider the following polynomials in $\fqs[Z]$:
$$
 \theta_1(Z)=(Z+\beta)^m+(Z-\beta)^m,\qquad  \theta_2(Z)=\frac{1}{\beta}((Z+\beta)^m-(Z-\beta)^m),
$$
Let
$$
t=\beta \frac{\bar t+1}{\bar t-1},
$$
As both $t$ and $\beta^2$ belong to $\fq$, 
the polynomials
\begin{equation}\label{hl}
h(Z)=t\theta_1(Z)+\beta^2\theta_2(Z),\qquad l(Z)=\theta_1(Z)+t\theta_2(Z)
\end{equation}
actually lie in $\fq[Z]$.
Taking into account that $ t=\beta\frac{\bar t+1}{\bar t-1}$, a straightforward computation gives
\begin{equation}\label{keykey}
\bar t\Big(\frac{Z+\beta}{Z-\beta}\Big)^m=\frac{\frac{h(Z)}{l(Z)}+\beta}{\frac{h(Z)}{l(Z)}-\beta}.
\end{equation}
Whence, 
$$
\bar t\Big(\frac{Z+\beta}{Z-\beta}\Big)^m+1=\frac{2h(Z)}{h(Z)-\beta l(Z)}\quad \text{and} \quad \bar t\Big(\frac{Z+\beta}{Z-\beta}\Big)^m-1=\frac{2\beta l(Z)}{h(Z)-\beta l(Z)}.
$$
We then have that $ M_{a,b,\bar t,m}(R,V)$ coincides with
$$
(R-\beta)^{2m}(V-\beta)^{2m}
\Big(\frac{2\beta l(R)}{h(R)-\beta l(R)}\Big)^2
\Big(\frac{2\beta l(V)}{h(V)-\beta l(V)}\Big)^2
g_{a,b}\left(\frac{h(R)}{l(R)},\frac{h(V)}{l(V)}\right).
$$
From
$$
h(Z)-\beta l(Z)=2(t-\beta)(Z-\beta)^m
$$
we obtain
$$
M_{a,b,\bar t,m}(R,V)=\frac{\beta^4}{(t-\beta)^4}
l(R)^2l(V)^2
g_{a,b}\left(\frac{h(R)}{l(R)},\frac{h(V)}{l(V)}\right),
$$
whence the assertion.
%\textcolor{blue}{
%Quindi in sostanza $M(R,V)$ si ottiene come
%$$
%g_{a,b}\left(\beta\frac{\frac{\frac{h(R)}{l(R)}+\beta}{\frac{h(R)}{l(R)}-\beta}+1}{\frac{\frac{h(R)}%{l(R)}+\beta}{\frac{h(R)}{l(R)}-\beta}-1},\beta\frac{\frac{\frac{h(V)}{l(V)}+\beta}{\frac{h(V)}{l(V)}-\beta}+1}{\frac{\frac{h(V)}{l(V)}+\beta}{\frac{h(V)}{l(V)}-\beta}-1}\right)=0
%$$
%che semplificando d\`a
%$$
%f(\frac{h(R)}{l(R)},\frac{h(V)}{l(V)})
%$$
%%Dopo aver eliminato i denominatori, dovrebbe essere in realt\`a
%$$
%M(R,V)=l(R)^2l(V)^2f\left(\frac{h(R)}{l(R)},\frac{h(V)}{l(V)}\right)
%$$
%}
\end{proof}

\begin{remark}\label{coordinate}
By the proof of Lemma {\rm{\ref{alldebolenuova}}}, 
for any $z\in \fq$, the $X$-coordinate of the point
$
Q_{\bar t(\frac{z+\beta}{z-\beta})^m }
$
is
$
u=\beta(\bar t(\frac{z+\beta}{z-\beta})^m +1)/(\bar t(\frac{z+\beta}{z-\beta})^m -1).
$
Then, by \eqref{keykey},  $u=\frac{h(z)}{l(z)}$ holds, with $h(Z)$ and $l(Z)$  as in \eqref{hl}. 
\end{remark}

\begin{remark}\label{notanodale}
If $(r,v)$ is an $\fq$-rational affine point of the curve $M_{a,b,\bar t,m}(R,V)=0$ with
$$
\Big(\frac{r+\beta}{r-\beta}\Big)^m\neq \Big(\frac{v+\beta}{v-\beta}\Big)^m
$$
then $P=(a,b)$ is collinear with 
 $Q_{\bar t (\frac{r+\beta}{r-\beta})^m}$ and $Q_{\bar t (\frac{v+\beta}{v-\beta})^m}$,
 which are two distinct points in $K_T$ by \eqref{descri}.
\end{remark}

\begin{proposition}\label{archinuovi} Let $P=(a,b)$ be a point in $AG(2,q)$ off $\cX$. Assume that \eqref{condizione} holds. If
$$
q+1-(6m^2-6m+2)\sqrt q \ge 4m^2+8m+1
$$
then $P$  is collinear with two distinct points of $K_T$.
\end{proposition}
\begin{proof}
Let $\K(\bar r, \bar v)$ be the function field of $M_{a,b,\bar t, m}(R,V)=0$, so that
$M_{a,b,\bar t,m}(\bar r,\bar v)=0$ holds. Let $E$ be the set of places $\gamma$ of $\K(\bar r,\bar v)$ for which at least one of the following holds:
\begin{itemize}
\item[(1)] $\gamma$ is a pole of either $\bar r$ or $\bar v$;

\item[(2)] $\gamma$ is a pole of either  $\Big(\frac{\bar r+\beta}{\bar r-\beta}\Big)$ or $\Big(\frac{\bar v+\beta}{\bar v-\beta}\Big)$;

\item[(3)] $\gamma$ is a zero of $\Big(\frac{\bar r+\beta}{\bar r-\beta}\Big)^m- \Big(\frac{\bar v+\beta}{\bar v-\beta}\Big)^m$.
\end{itemize}
As both degrees of the extensions $\K(\bar r,\bar v):\K(\bar r)$ and 
$\K(\bar r,\bar v):\K(\bar v)$ are equal to $2m$, the number of places satisfying (1) is at most $4m$. According to the proof of Lemma \ref{irrmodificato}, we have that
$$
\bar x= \frac{\bar r+\beta}{\bar r-\beta},\qquad \bar y=\frac{\bar v+\beta}{\bar v-\beta}
$$
satisfy $f_{A,B,\bar t, m}(\bar x,\bar y)=0$. Therefore, 
by Propositions \ref{clear} and \ref{clear2}
the number places satysfying (2) is $4m$. 
It is easily seen that in $\K(\bar u, \bar z)$ the rational function $\bar u-\bar z$ has at most $4$ distinct zeros; hence, the set of poles of $\bar x^m-\bar y^m$ in $\K(\bar x,\bar y)$ has size less than or equal to $4m^2$. This shows that $E$ comprises at most $4m^2+8m$ places.
Our assumption on $q$ and $m$, together with the Hasse-Weil bound, ensures the existence of at least $4m^2+8m+1$ $\fq$-rational places of $\K(\bar r,\bar v)$; hence,
there exists at least one $\fq$-rational place $\gamma_0$ of $\K(\bar r,\bar s)$ not in $E$.
Let $\tilde r=\bar r(\gamma_0)$ and $ \tilde v=\bar v(\gamma_0).$
By Remark \ref{notanodale},
$P=(a,b)$ is collinear with 
 $Q_{\bar t (\frac{\tilde r+\beta}{\tilde r-\beta})^m}$ and $Q_{\bar t (\frac{\tilde v+\beta}{\tilde v-\beta})^m}$,
 which are two distinct points in $K_T$.
\end{proof}

The following technical variant of Proposition \ref{archinuovi} will also be needed.
\begin{proposition}\label{archinuoviVAR}  Let $P=(a,b)$ be a point in $AG(2,q)$ off $\cX$. Assume that \eqref{condizione} holds. If
\begin{equation}\label{aritmeticaVAR}
q+1-(6m^2-6m+2)\sqrt q \ge 8m^2+8m+1
\end{equation}
then $P$  is collinear with two distinct points of $K_T\setminus \{T\}$.
\end{proposition}
\begin{proof}
One can argue as in the proof of Proposition \ref{archinuovi}.
We need to ensure that neither 
 $Q_{\bar t (\frac{\tilde r+\beta}{\tilde r-\beta})^m}$ or $Q_{\bar t (\frac{\tilde v+\beta}{\tilde v-\beta})^m}$ coincides with $T$. As $T=Q_{\bar t}$, this is equivalent to $\gamma_0$ not being a zero of either $(\frac{\bar r+\beta}{\bar r-\beta})^m-1$ or $(\frac{\bar v+\beta}{\bar v-\beta})^m-1$ in the function field $\K(\bar r,\bar v)$. By Proposition \ref{clear}, in  $\K(\bar u,\bar z)$ both rational functions $\bar u-1$ and $\bar z-1$ have at most two distinct zeros. Therefore, there are at most $4m^2$ places $\gamma_0$ that need to be ruled out.
\end{proof}

If \eqref{condizione} is not satisfied, then $P$ is not collinear with any two points of $K_T$. Actually, a stronger statement holds.
\begin{proposition}\label{fuori} Let $a,b \in \fq$ be such that
$$
(a,b)\in \Big\{(0,-\frac{9}{8\beta^2}),(\beta\sqrt{-3},0),(-\beta\sqrt{-3},0)\Big\}.
$$
Then the point $P=(a,b)$ is not collinear with any two  $\fq$-rational affine points of $\cX$.
\end{proposition}
\begin{proof}
We recall that by the proof of Lemma \ref{alldebolenuova},
 the point $P$ is collinear with $(x,\frac{1}{x^2-\beta^2})$ and $(y,\frac{1}{y^2-\beta^2})$, with $x,y\in \fq$, if and only if 
$g_{a,b}(x,y)=0$.
If
$
(a,b)=(0,-\frac{9}{8\beta^2})
$
then
$$g_{a,b}(X,Y)=-\frac{1}{8\beta^2}(9X^2Y^2-\beta^2(X^2+Y^2)+8\beta^2XY+\beta^4)=$$

$$=-\frac{1}{8\beta^2}(3XY - X\beta + Y\beta + \beta^2)(3XY + X\beta -Y\beta + \beta^2).
$$
If $g_{a,b}(x,y)=0$, then either
\begin{equation}\label{9feb}
3xy - x\beta + y\beta + \beta^2=0\quad  \text{ or }\quad 3xy + x\beta -y\beta + \beta^2=0.
\end{equation}
If $(x,y)\in \fq$, then both $x$ and $y$ are fixed by the Frobenius map over $\fq$, and hence both equalities in \eqref{9feb} hold. This easily implies $x=y$. Then no two distinct $\fq$-rational affine points of $\cX$ can be collinear with $(a,b)$.

Note that  $(a,b)=(\pm \beta \sqrt{-3},0)$  can only occur when $-3$ is a non-square in $\fq$, otherwise $\pm\beta\sqrt{-3}\notin \fq$.  In this case, $(\sqrt{-3})^q=-\sqrt{-3}$ holds; also, 
%$-3$ is not a square in $\fq$ 
%and  $\beta$ is chosen so that  $\beta^2=-3$ then
%$$
%g_{-3,0}(X,Y)=-(X^2+Y^2) - XY -3(X + Y) -3=\\-(x+\frac{1+\beta}{2}y+\frac{3+\beta}{2})
%(x+\frac{1-\beta}{2}y+\frac{3-\beta}{2})
%$$
%and similarly
%$$
%g_{3,0}(X,Y)=-(X^2+Y^2) - XY +3(X + Y) -3=\\-(x+\frac{1-\beta}{2}y-\frac{3-\beta}{2})
%(x+\frac{1+\beta}{2}y-\frac{3+\beta}{2})
%$$
%Se $\beta$ \`e diverso da $\sqrt{-3}$ allora bisogna fattorizzare $g_{\pm\beta \sqrt{-3},0}(X,Y)$. Proviamo:

$$
g_{\beta\sqrt{-3},0}(X,Y)=-(X^2+Y^2) - XY +\beta\sqrt{-3}(X + Y) +\beta^2=
$$

$$
=-\Big(X+\frac{1+\sqrt{-3}}{2}Y+\frac{-\beta\sqrt{-3}+\beta}{2}\Big)\Big(X+\frac{1-\sqrt{-3}}{2}Y+\frac{-\beta\sqrt{-3}-\beta}{2}\Big)
$$
and
$$
g_{-\beta\sqrt{-3},0}(X,Y)=-(X^2+Y^2) - XY -\beta\sqrt{-3}(X + Y) +\beta^2=
$$

$$
=-\Big(X+\frac{1-\sqrt{-3}}{2}Y+\frac{\beta\sqrt{-3}+\beta}{2}\Big)\Big(X+\frac{1+\sqrt{-3}}{2}Y+\frac{\beta\sqrt{-3}-\beta}{2}\Big)
$$
The assertion for $(a,b)=(\pm \beta \sqrt{-3},0)$ then follows by the same arguments used for $(a,b)= (0,-\frac{9}{8\beta^2})$.
\end{proof}

In order to investigate the bicovering properties of the arc $K_t$, according to Remark \ref{coordinate}
we need to consider the rational function $
\big(a-\frac{h(\bar r)}{l(\bar r)}\big)\big(a-\frac{h(\bar v)}{l(\bar v)}\big)
$
in the function field of $M_{a,b,\bar t,m}(R,V)=0$.
\begin{lemma}\label{exitre}
Let $P=(a,b)$ be a point in $AG(2,q)$ off $\cX$ satisfying \eqref{condizione}.
Let $\K(\bar r, \bar v)$ be the function field of $M_{a,b,\bar t, m}(R,V)=0$, so that
$M_{a,b,\bar t,m}(\bar r,\bar v)=0$. Then the rational function
$
\big(a-\frac{h(\bar r)}{l(\bar r)}\big)\big(a-\frac{h(\bar v)}{l(\bar v)}\big)
$
is not a square in $\K(\bar r,\bar v)$.
\end{lemma}
\begin{proof} Let $\bar x$ and $\bar y$ be as in the proof of Proposition \ref{archinuovi}, so that
$\K(\bar r,\bar v)=\K(\bar x,\bar y)$ with $f_{A,B,\bar t, m}(\bar x,\bar y)=0$. By straightforward computation,
$$
\Big(a-\frac{h(\bar r)}{l(\bar r)}\Big)\Big(a-\frac{h(\bar v)}{l(\bar v)}\Big)
=\frac{4\beta^2(\bar t \bar x^m-A)(\bar t \bar y^m-A)}{(A-1)^2(\bar t \bar x^m-1)(\bar t \bar y^m-1)} .
$$
Then the assertion follows from Proposition \ref{exiisolated}.
\end{proof}
\begin{proposition}\label{archibishop} Let $P=(a,b)$ be a point in $AG(2,q)$ off $\cX$. Assume that \eqref{condizione} holds. If
\begin{equation}\label{aritmetica2}
q+1-(16m^2-8m+2)\sqrt q \ge 16m^2+24m+1
\end{equation}
then $P$  is bicovered by the points of $K_T$.
\end{proposition}
\begin{proof} Let $\K(\bar r, \bar v)$ be the function field of $M_{a,b,\bar t, m}(R,V)=0$, so that
$M_{a,b,\bar t,m}(\bar r,\bar v)=0$. By Proposition \ref{exiisolated} and Lemma \ref{exitre}, for every $c\in \fq^*$
the equation 
$$
\bar w^2=c \Big(a-\frac{h(\bar r)}{l(\bar r)}\Big)\Big(a-\frac{h(\bar v)}{l(\bar v)}\Big)
$$
defines a Kummer extension $\K(\bar r,\bar v,\bar w)$ of $\K(\bar r, \bar v)$
with genus less than or equal to $8m^2-4m+1$. Let $E$ be as in the proof of Proposition \ref{archinuovi}, and let $E'$ be the set of places of $\K(\bar r,\bar v,\bar w)$ that either lie over a place in $E$ or over a zero or a pole of $\big(a-\frac{h(\bar r)}{l(\bar r)}\big)\big(a-\frac{h(\bar v)}{l(\bar v)}\big)$. By Proposition \ref{exiisolated}, together with the proof of Proposition \ref{archinuovi}, an upper bound for the size of $E'$ is $16m^2+24m$. 
 Our assumption on $q$ and $m$, together with Proposition \ref{HaWe}, ensures the existence of at least $16m^2+24m+1$ $\fq$-rational places of $\K(\bar r,\bar v,\bar w)$; hence,
there exists at least one $\fq$-rational place $\gamma_c$ of $\K(\bar r,\bar v,\bar w)$ not in $E'$.
Let 
$$
\tilde r=\bar r(\gamma_c),\quad \tilde v=\bar v(\gamma_c), \quad \tilde w=\bar w(\gamma_c).
$$
Note that $P_c=(\tilde r,\tilde v)$ is an $\fq$-rational affine point of the curve with equation $M_{a,b,\bar t,m}(R,V)=0$. Therefore, by Remark \ref{notanodale}, 
$P$ is collinear with two distinct points
$$
P_{1,c}=Q_{\bar t (\frac{\tilde r+\beta}{\tilde r-\beta})^m},\,P_{2,c}=Q_{\bar t (\frac{\tilde v+\beta}{\tilde v-\beta})^m}\in K_T.
$$
If $c$ is chosen to be a square, then $P$ is external to $P_{1,c}P_{2,c}$; on the other hand, if $c$ is not a square, then $P$ is internal to $P_{1,c}P_{2,c}$. 
This proves the assertion.
\end{proof}

In the final part of this section we deal with points in $\cX$.
%Assume that $K_{T'}$ is a coset of $K$ such that $K_T\cup K_{T'}$ is an arc, and let %$P_u=(u,\frac{1}{u^2-\beta^2})$ be an $\fq$-rational affine point of $\cX$ not %belonging to $K_T\cup K_{T'}$ but collinear with a point of $K_T$ and a point of %$K_{T'}$.
\begin{proposition}\label{bicointerni} 
Let $K_{T'}$ be a coset of $K$ such that $K_T\cup K_{T'}$ is an arc. For $u\in \fq$, let $P_u=(u,\frac{1}{u^2-\beta^2})$ be an $\fq$-rational affine point of $\cX$ not belonging to $K_T\cup K_{T'}$ but collinear with a point of $K_T$ and a point of $K_{T'}$.
\begin{itemize}
\item[{\rm{(i)}}] If $u\neq 0$ and \eqref{aritmetica2} holds, then $P_u$ is bicovered by $K_T\cup K_{T'}$.
\item[{\rm{(ii)}}] The point $P_0=(0,-\frac{1}{\beta^2})$ is not bicovered by $K_T\cup K_{T'}$. It is  internal (resp. external) to every segment cut out on $K_T\cup K_{T'}$ by a line through $P_0$  when $q\equiv 1\pmod 4$ (resp. $q\equiv 3\pmod 4$).
\end{itemize}
\end{proposition}
\begin{proof}
Note that when $P$ ranges over  $K_T$, then the point $Q=\ominus (P_u\oplus P)$  ranges over $K_{T'}$ and is collinear with $P_u$ and $P$.  Recall that  $P$ belongs to $K_T$ if and only if $P=(e,\frac{1}{e^2-\beta^2})$ with 
$$
e=\beta\frac{\bar t(\frac{x+\beta}{x-\beta})^m+1}{\bar t(\frac{x+\beta}{x-\beta})^m-1}
$$ 
 for some $x\in \fq$. In this case,  $Q=(s(e),\frac{1}{s(e)^2-\beta^2})$ with
$
s(e)=-\frac{ue+\beta^2}{u+e}.
$
%=-\frac{u\beta\frac{\bar t(\frac{z+\beta}{z-\beta})^m+1}{\bar t(\frac{z+\beta}{z-\beta})^m-1}+\beta^2}%{u+\beta\frac{\bar t(\frac{z+\beta}{z-\beta})^m+1}{\bar t(\frac{z+\beta}{z-\beta})^m-1}}.
%$$
For an element $\bar x$ transcendental over $\K$ let 
$$
e(\bar x)=\beta\frac{\bar t(\frac{\bar x+\beta}{\bar x-\beta})^m+1}{\bar t (\frac{\bar x+\beta}{\bar x-\beta})^m-1}=\frac{\beta\bar t (\bar x+\beta)^m+\beta(\bar x-\beta)^m}{\bar t (\bar x+\beta)^m-(\bar x-\beta)^m}\in \K(\bar x).
$$
Note that $e(\bar x)$ is defined over $\fq$.
In order to determine whether $P_u$ is bicovered by $K_T\cup K_{T'}$ we need to investigate whether the following rational function is a square in $\K(\bar x)$:
$$
\eta(\bar x)=(u-e(\bar x))(u-s(e(\bar x)))=
%(u-e)(u+\frac{ue(\bar x)+\beta^2}{u+e(\bar x)})=
\frac{u-e(\bar x)}{u+e(\bar x)}(u^2+2ue(\bar x)+\beta^2)
%=\frac{u-e(\bar x)}{u+e(\bar x)}((u+e(\bar x))^2+\beta^2-e(\bar x)^2)
$$
Let $\gamma $ be a zero of $\bar t(\frac{\bar x+\beta}{\bar x-\beta})^m-1$ in $\K(\bar x)$. Note that since $(m,p)=1$, the polynomial $tZ^m-1$ has no multiple roots in $\K[Z]$. Then the valuation $v_\gamma(e(\bar x))$ of $e(\bar x)$ at $\gamma$ is $-1$.  If in addition $u\neq 0$, then
$
v_\gamma(\eta(\bar x))=v_\gamma(e(\bar x))=-1,
$
whence $\eta(\bar x)$ is not a square in $\K(\bar x)$ and Proposition \ref{teo1cor} applies to $c\eta(\bar x)$ for each $c\in \fq^*$. Since the number of poles of $\eta(\bar x)$ is at most $2m$,  the genus of the Kummer extension $\K(\bar x,\bar w)$ of $\K(\bar x)$ with $\bar w^2=c\eta(\bar x)$ is at most $2m-1$.

Our assumption on $q$, together with the Hasse-Weil bound, yield the existence of an $\fq$-rational place $\gamma_c$ of $\K(\bar x,\bar w)$ which is not a zero nor a pole of $\bar w$. 
Let 
$
\tilde x=\bar x(\gamma_c)$, $ \tilde w=\bar w(\gamma_c)
$,
$$
\tilde e =\beta\frac{\bar t(\frac{\tilde x+\beta}{\tilde x-\beta})^m+1}{\bar t(\frac{\tilde x+\beta}{\tilde x-\beta})^m-1}\quad \text{and}\quad
s(\tilde e)=-\frac{u\tilde e+\beta^2}{u+\tilde e}.
$$
Therefore, if $u\neq 0$, then
$P_u$ is collinear with two distinct points 
$$
P(c)=\Big(\tilde e,\frac{1}{\tilde e^2-\beta^2}\Big)\in K_T\qquad
Q(c)=\Big(s(\tilde e),\frac{1}{s(\tilde e)^2-\beta^2}\Big)\in K_{T'}.
$$
If $c$ is chosen to be a square, then $P_u$ is external to $P(c)Q(c)$; on the other hand, if $c$ is not a square, then $P_u$ is internal to $P(c)Q(c)$. 

Assume now that $u=0$. First note that $P_0$ coincides with $Q_{-1}$, and hence  belongs to $K$. Therefore, as $m$ is odd, $P_0$ cannot be collinear with any two points from the same coset of $K$. 
Assume then that $P_0$ is collinear with $P=(e,\frac{1}{e^2-\beta^2})\in K_T$ 
and $Q=\big(s(e),\frac{1}{s(e)^2-\beta^2}\big)\in K_{T'}$.
It is straightforward to check that $(u-e)(u-s(e))=e\cdot s(e)=-\beta^2$. Since $\beta^2$ is not a square in $\fq$,
 the assertion follows from the well-known fact that $-1$ is a square in $\fq$ precisely when $q\equiv 1 \pmod 4$. 
\end{proof}

\section{Complete arcs and complete caps from cubics with an isolated double point}

Throughout this section $q=p^s$ with $p$ a prime, $p>3$. Also, $\cX$, $G$, $m$, $K$ and $K_T$ are as in Section \ref{covvi}.

We recall the notion of a maximal-$3$-independent subset of a finite abelian group $\mathcal G$,  as given in \cite{MR1075538}.
A subset $M$ of $\mathcal G$ is said to be {\em maximal }$3$-{\em independent} if 
\begin{itemize}
\item[ (a)] $x_1+x_2+x_3\neq 0$ for all $x_1,x_2,x_3\in M$, and 
\item[(b)] for each $y\in \mathcal G\setminus M$ there exist $x_1,x_2\in M$ with $x_1+x_2+y=0$. 
\end{itemize}
If in (b) $x_1\neq x_2$ can be assumed, then $M$ is said to be {\em good}.

Assume that $S$ is a good maximal $3$-independent subset of $G$.
Since three points in $G$ are collinear if and only if their sum is equal to the neutral element,
 $S$ is an arc whose secants cover all the points in $G$. 

For  direct products of abelian groups  of order at least $4$, an explicit construction of good maximal $3$-independent subsets was provided by  Sz\H onyi; see e.g. \cite[Example 1.2]{MR1221589}. If $m$ and $(q+1)/m$ are coprime, such a construction applies to $G$. 

\begin{proposition}\label{mazzi3i} Assume that $m$ and $(q+1)/m$ are coprime. Let $H$ be the subgroup of $G$ of order $m$, so that $G$ is the direct product of $K$ and $H$.  Fix 
two elements $R\in K$ and $R'\in H$ of order greater than $3$, and let $T=R'\ominus 2R$.
Then
$$
\mathcal A=K_T\setminus \{T\} \quad \bigcup \quad (H\oplus R) \setminus \{\ominus  2T\oplus R\}
$$
is a good maximal $3$-independent subset of $G$.
\end{proposition}

Let $\mathcal E$  denote the set of points $P$ in $AG(2,q)\setminus \cX$ whose affine coordinates $(a,b)$ do not satisfy \eqref{condizione}. By Remark \ref{quattordici}, the size of $\mathcal E$ is $3$ precisely when $s$ is odd and $p\equiv 2 \pmod 3$; otherwise, $\mathcal E$ consists of the point with coordinates $(0,-\frac{9}{8\beta^2})$.

\subsection{Small complete arcs in $AG(2,q)$}
Let $\mathcal A$ be as in Proposition \ref{mazzi3i}.
We use Propositions \ref{archinuoviVAR}, \ref{fuori}, and \ref{mazzi3i} in order to construct small complete arcs in Galois planes.
Note that \eqref{aritmeticaVAR}  is implied by
$
m\le \frac{\sqrt[4]q}{\sqrt 6}.
$
\begin{theorem}\label{unouno} Let $q=p^s$ with $p>3$ a prime. Let $m$ be a divisor of $q+1$ such that $(m,6)=1$ and $(m,\frac{q+1}{m})=1$. If $m\le \frac{\sqrt[4]q}{\sqrt 6}$, then 
\begin{itemize}
\item if either $s$ is even or $p\equiv 1 \pmod 3$, the set $\mathcal A\cup \mathcal E$ is a complete arc in $AG(2,q)$ of size $m+\frac{q+1}{m}-2$;

\item if $s$ is odd and $p\equiv 2 \pmod 3$, the set
$\mathcal A\cup \mathcal E$ contains a complete arc in $AG(2,q)$ of size at most $m+\frac{q+1}{m}$.
\end{itemize}
\end{theorem}

\subsection{Small complete caps in $AG(N,q)$, $N\equiv 0\pmod 4$}

Let $M$ be a maximal $3$-independent subset of the factor group $G/K$ containing $K_T$. Then
the union $S$ of the cosets of $K$ corresponding to  $M$  is a good maximal $3$-independent subset of $G$; see \cite{MR1075538}, Lemma 1, together with Remark 5(5). It has already been noticed that $S$ is an arc whose secants cover all the points in $G$. Note also that $K$ is disjoint from $S$, and hence the point $P_0=(0,-\frac{1}{\beta^2})$ does not belong to $S$.

If either $s$ is even or $p\equiv 1 \pmod 3$, by Propositions \ref{fuori}, \ref{archibishop}, and \ref{bicointerni},  then $S\cup \{(0,-\frac{9}{8\beta^2})\}$ is an almost bicovering arc with center $P_0$, provided that $m$  is small enough with respect to $q$. 

\begin{theorem}\label{pre14mar} Let $q=p^s$ with $p>3$ a prime, and assume that either $s$ is even or $p\equiv 1 \pmod 3$. Let $m$ be a proper divisor of $q+1$ such that $(m,6)=1$ and \eqref{aritmetica2}
holds. Let $K$ be the subgroup of $G$ of index $m$. For $M$  a maximal $3$-independent subset of the factor group $G/K$, the point set
\begin{equation}\label{13maggio}
\mathcal B=\Big(\bigcup_{K_{T_i}\in M}K_{T_i}\Big)\bigcup \mathcal E
\end{equation} 
is an almost bicovering arc in $AG(2,q)$ with center $P_0=(0,-\frac{1}{\beta^2})$. The size of $\mathcal B$ is $\#M\cdot \frac{q+1}{m}+1$.
\end{theorem}

When $s$ is odd and $p\equiv 2 \pmod 3$  a further condition on $M$ is needed in order to ensure that $\mathcal B$ as in \eqref{13maggio} is an almost bicovering arc. Note that by Proposition \ref{fuori} there is precisely one point 
 in $G$ collinear with any two points in $\mathcal E$. 
%$(0,-\frac{9}{8\beta^2})$ and  $(\beta\sqrt{-3},0)$; similarly, only one point in $G$ is %collinear with $(0,-\frac{9}{8\beta^2})$ and  $(-\beta\sqrt{-3},0)$.

\begin{theorem}\label{pre14marBIS} Let $q=p^s$ with $p>3$ a prime. Assume that $s$ is odd and $p\equiv 2 \pmod 3$.  Let $m$ be a proper divisor of $q+1$ such that $(m,6)=1$ and \eqref{aritmetica2}
holds. Let $K$ be the subgroup of $G$ of index $m$. 
Let $Q_1$ denote the only point in $G$ collinear with 
$(0,-\frac{9}{8\beta^2})$ and  $(\beta\sqrt{-3},0)$; similarly, let $Q_2\in G$ be collinear with $(0,-\frac{9}{8\beta^2})$ and  $(-\beta\sqrt{-3},0)$.
For $M$  a maximal $3$-independent subset of the factor group $G/K$ not containing $K\oplus Q_1$ nor $K\oplus Q_2$, the point set
$$
\mathcal B=\Big(\bigcup_{K_{T_i}\in M}K_{T_i}\Big)\bigcup \mathcal E
$$ 
is an almost bicovering arc in $AG(2,q)$ with center $P_0=(0,-\frac{1}{\beta^2})$. The size of $\mathcal B$ is $\#M\cdot \frac{q+1}{m}+3$.
\end{theorem}

We use Theorems \ref{pre14mar} and \ref{pre14marBIS}, together with Proposition \ref{mainP}, in order to construct small complete caps in affine spaces $AG(N,q)$.
Assume that $m=m_1m_2$ with $(m_1,m_2)=1$. Then the factor group $G/K$ is the direct product of two subgroups of order $m_1>4$ and $m_2>4$, and the aforementioned construction by Sz\H onyi \cite[Example 1.2]{MR1221589} of a maximal $3$-independent set $M$ of size $m_1+m_2-3$ applies. It is easily seen that $M$ can be chosen in such a way that it does not contain any two fixed cosets of $K$. 
As \eqref{aritmetica2}  is implied by
$
m\le \frac{\sqrt[4]q}{4},
$
 the following result holds. 
\begin{theorem}\label{duedue}
Let $q=p^h$ with $p>3$, and 
let $m$ be a proper divisor of $q+1$ such that $(m,6)=1$ and 
$m\le \frac{\sqrt[4]q}{4}$. Assume that  $m=m_1m_2$ with $(m_1,m_2)=1$. Then  
for $N\equiv 0 \pmod 4$, $N\ge 4$, there exists a complete cap in $AG(N,q)$ of size less than or equal to
$$
\Big((m_1+m_2-3)\cdot \frac{q+1}{m}+3\Big)q^{\frac{N-2}{2}}.
$$
\end{theorem}

%Then Theorem  \ref{duedue} follows at once from Corollary \ref{lastbut1}.

\end{document}